\documentclass[12pt,reqno]{amsart}

\usepackage[colorlinks]{hyperref}
\usepackage{amssymb}
\usepackage{mathrsfs}
\usepackage[all,cmtip]{xy}
\usepackage{cite}
\usepackage{color}
\usepackage{tikz-cd}
\usepackage{framed}

\newtheorem{theorem}{Theorem}[section]

\newtheorem{lemma}[theorem]{Lemma}
\newtheorem{proposition}[theorem]{Proposition}

\newtheorem*{Theorem A}{Theorem A}
\newtheorem*{Theorem B}{Theorem B}
\newtheorem*{problem NOP}{Contractive Projection Problem}
\newtheorem*{problem A}{Problem A}
\newtheorem*{problem B}{Problem B}
\theoremstyle{definition}

\numberwithin{equation}{section}

\newcommand{\D}{\mathbb D}

\newcommand{\C}{\mathbb C}

\newcommand{\T}{\mathbb T}
\newcommand{\E}{\mathbb E}
\newcommand{\one}{\mathbf 1}
\newcommand{\cal}{\mathcal}
\newcommand{\bb}{\mathbb}

\newcommand{\ol}{\overline}
\newcommand{\Ran}{\textrm{Ran}\,}

\allowdisplaybreaks

\title[Contractive projections on $H^p$ spaces $(0<p<1)$]{Contractive projections, conditional expectations, and idempotent coefficient multipliers on $H^p$ spaces $(0<p<1)$}

\author{Xiangdi Fu}
\address{Xiangdi Fu: School of Fundamental Physics and Mathematical Sciences, HIAS, University of Chinese Academy of Sciences, Hangzhou, 310024, China}
\email{xdfu@ucas.ac.cn}

\author{Kunyu Guo}
\address{Kunyu Guo: School of Mathematical Sciences, Fudan University, Shanghai, 200433, China}
\email{kyguo@fudan.edu.cn}

\author{Dilong Li}
\address{Dilong Li: School of Mathematical Sciences, Fudan University, Shanghai, 200433, China}
\email{dlli23@m.fudan.edu.cn}


\makeatletter
\@namedef{subjclassname@2020}{\textup{2020} Mathematics Subject Classification}
\makeatother

\begin{document}
	\begin{abstract}
		 In this paper, we investigate contractive projections, conditional expectations, and idempotent coefficient multipliers on the Hardy spaces $H^p(\T)$ for $0<p<1$. For such values of $p$, we first establish a general extension theorem for contractive projections in a probability $L^p$-space. Combining this theorem with the study of conditional expectations on $H^p(\T)$, we characterize a broad class of contractive projections on $H^p(\T)$ that are of particular interest. Furthermore, we apply these results to give a complete characterization of contractive idempotent coefficient multipliers for the Hardy spaces $H^p(\T^d)$ on the $d$-dimensional torus for $0<p<1$ and $1\leq d\leq \infty$. This complements a remarkable result of Brevig, Ortega-Cerd\`{a}, and Seip characterizing such multipliers on $H^p(\T^d)$ for $1\leq p \leq \infty$.
		\end{abstract}
	\maketitle

	\section{Introduction}

	\subsection{A Short Review of the Contractive Projection Problem}

	Let $X$ be a (quasi-)Banach space with (quasi-)norm $\|\cdot \|$. A bounded linear operator $P: X\to X$ is called a projection if $P^2=P$. We always assume that $P\neq 0$, and then $$\|P\|:=\sup_{\|x\|=1} \|Px\| \geq 1.$$ By a {\it contractive projection}, we mean a nonzero projection $P$ with $\|P\|=1$. Contractive projections can be regarded as a generalization of orthogonal projections on Hilbert spaces to the setting of Banach and quasi-Banach spaces, and they constitute a fundamental object of study in approximation theory and in the isometric theory of these spaces. Thus, the following problem arises naturally.

\begin{problem NOP}
Given a (quasi-)Banach space $X$, can we obtain an explicit description of all contractive projections on $X$?
\end{problem NOP}

	However, it seems hopeless to resolve the above problem for an arbitrary space \( X \), as even for certain classical spaces, the known results are far from trivial.

	Below, we list some related references that address the contractive projection problem on various spaces \( X \):
\begin{itemize}
    \item \( X = L^1(\mu) \), with \( \mu \) a probability measure; see R. G. Douglas \cite{Do}.
    \item \( X = L^p(\mu) \), with \( \mu \) a probability measure and \( 0 < p < \infty, p\neq 1,2\); see T. And\^o \cite{An}.
    \item $X=H^\infty(\T)$ and $X=A(\T)$; see P. Wojtaszczyk \cite{Wo79}. 
    \item $X=H^p(\T)$ with $1\leq p<\infty$; see F. Lancien, B. Randrianantoanina, and E. Ricard \cite{LRR}, as well as a recent preprint of the authors \cite{FGL}.
    \item $X= C_0(K)$, with $K$ a locally compact Hausdorff space; see Y. Friedman and B. Russo \cite{FR}.
    \item $X=\text{Sequence spaces}$; see B. Randrianantoanina \cite{Ran04,Ran98,Ran}. 
    \item $X=\text{Paley-Wiener spaces}$; see A. Kulikov and I. Zlotnikov \cite{KZ}. 
    \item $X= \cal C_\infty$, the space of compact operators on a separable complex Hilbert space, and $X=\cal C_p$ the von Neumann-Schatten class with $1\leq p<\infty$; see J. Arazy and Y. Friedman \cite{AF,AF+}.
    \item $X=\text{vector-valued function spaces}$; see Y. Raynaud \cite{Ray} and B. Lemmens, B. Randrianantoanina and O. van Gaans \cite{LRG07}.
\end{itemize} 

We do not attempt to provide a complete account of all relevant references due to the sheer volume of literature on these topics. Nonetheless, we believe that the excellent survey by B. Randrianantoanina \cite{Ran} offers valuable historical insights and a detailed bibliography on these topics. 

It is worth mentioning that, recently, O. F. Brevig, J. Ortega-Cerd\`{a}, and K. Seip \cite{BOS} completely characterized the contractive projections induced by idempotent coefficient multipliers on $H^p(\T^d)$ for $1\leq p\leq \infty$ and $1\leq d\leq \infty$. Here $\T^d$ stands for the $d$-dimensional torus (countably infinite when $d=\infty$), endowed with the normalized Haar measure $m_d$. For $0<p<\infty$, the Hardy space $H^p(\T^d)$ is defined as the closure of analytic polynomials in $L^p(\T^d)$. Let $\mathbb Z^{(d)}:=\bigoplus_{j=1}^d \mathbb Z$ be the dual group of $\T^d$ and let $\mathbb N_0^{(d)}:=\bigoplus_{j=1}^d \mathbb N_0$ be the positive cone of $\mathbb Z^{(d)}$. For a subset $\Gamma$ of $\mathbb N_0^{(d)}$, the associated idempotent coefficient multiplier $P_\Gamma$ is defined on the space of analytic polynomials by
\begin{align}\label{eq1.1}
  P_{\Gamma} : \,\,\sum_{\alpha \in \mathbb N_0^{(d)}} c_\alpha z^\alpha \mapsto \sum_{\alpha \in \Gamma} c_\alpha z^\alpha.
\end{align}
Here, multi-indices are used in the usual way: for $z=(z_j)_{j=1}^d \in \mathbb \T^d$ and $\alpha=(\alpha_j)_{j=1}^d \in \mathbb N_0^{(d)}$, $z^\alpha:= \prod_{j=1}^d z_j^{\alpha_j}.$ Specifically, it is shown in \cite{BOS} that when $1\leq p\leq \infty$ is not an even integer, $P_\Gamma$ extends to a contraction on $H^p(\T^d)$ if and only if $\Gamma=\Lambda\cap \mathbb N^{(d)}_0$ for some coset $\Lambda \subseteq \mathbb Z^{(d)}$. Moreover, for even integers $p=2k\geq 4$, the contractivity of $P_\Gamma$ on $H^{2k}(\T^d)$ depends on $k$, $d$, and the affine dimension of $\Gamma$ in an interesting way; see \cite[Theorem 1.2]{BOS} for details. Actually, they provided an effective algorithm \cite[Theorem 1.3]{BOS} to examine whether $P_\Gamma$ extends to a contractive projection on $H^{2k}(\T^d)$. 

The remarkable results in \cite{BOS} cover the case $1\leq p\leq \infty$, while the range $0<p<1$ remains unknown. This inspires us to investigate contractive projections on Hardy spaces $H^p(\T^d)$ for $0<p<1$ and $1\leq d\leq \infty$. To this aim, we will first study contractive projections on $H^p(\T)$ that leave the constants intact, which are closely connected to the conditional expectations associated with the $\sigma$-algebras generated by inner functions. Then we apply these results to provide a complete characterization of contractive idempotent coefficient multipliers on $H^p(\T^d)$ for $0<p<1$ and $1\leq d\leq \infty$. 

We mention that, in a recent preprint\cite{FGL}, the authors also provided explicit expressions for all contractive projections on $H^p(\T)$ for $1\leq p<\infty$, which requires truly different approaches.

\subsection{Conditional expectations and the Douglas-And\^o theorem} In this subsection we review some basic properties of conditional expectations and the solution to the Contractive Projection Problem for $L^p(\mu)=L^p(\Omega, {\bf \Sigma}, \mu)$. Here and throughout the paper, $(\Omega, \mathbf{\Sigma}, \mu)$ stands for a complete probability space. Suppose $\mathbf{\Sigma}'$ is a complete sub $\sigma$-algebra of $\bf{\Sigma}$. The conditional expectation associated with $\mathbf{\Sigma}'$ is an operator $$\E(\cdot| \mathbf{\Sigma}'): L^1(\Omega, {\bf \Sigma},\mu) \to L^1(\Omega, \mathbf{\Sigma}', \mu); \quad f\mapsto \E(f| \mathbf{\Sigma}'),$$ uniquely determined by the following identities:
	\begin{align*}
	\int_{F} \E(f| \mathbf{\Sigma}') d\mu = \int_{F} f d\mu, \quad  \forall F\in \mathbf{\Sigma}'.
	\end{align*}
	Indeed, the existence and the uniqueness of $\E(f| \mathbf{\Sigma}')$ follow from the Radon-Nikodym theorem. The conditional expectation $\E(\cdot| \mathbf{\Sigma}')$ is a contractive projection from $L^1(\Omega,\bf{\Sigma}, \mu )$ onto its closed subspace $L^1(\Omega,\mathbf{\Sigma}', \mu )$. By Jensen's inequality, $\mathbb{E}(\cdot | \mathbf{\Sigma}')$ is also contractive with respect to the $L^p$-norm for all $1 \leq p \leq \infty$. In particular, when $p=2$, $\E(\cdot| \mathbf{\Sigma}')$ is the orthogonal projection from $L^2(\Omega, \bf{\Sigma}, \mu)$ onto $L^2(\Omega, \mathbf{\Sigma}', \mu)$. The conditional expectation, introduced by A. Kolmogorov in 1933, is a fundamental concept in probability theory.  We refer to \cite{Du} for further properties and applications of conditional expectation. 
	
	As well known, for any complete $\sigma$-algebras $\mathbf{\Sigma}'$, the corresponding conditional expectations $\E(\cdot |\mathbf{\Sigma}')$ is a contractive projection on $L^p(\Omega, \bf{\Sigma}, \mu)$ where $1\leq p\leq \infty$, which leave constants intact. In his seminal paper \cite{Do}, Douglas showed that when $p=1$, this operator-theoretic property characterizes conditional expectations. Later, T. And\^o proved that Douglas's result also holds true for $1<p<\infty$, $p\neq 2$. We reformulate their result below as the Douglas-And\^o theorem.
	
	\begin{Theorem A}[The Douglas-And\^o Theorem]\label{1.1}
	Let $1\leq p<
	\infty$, $p\neq 2$ and let $(\Omega, {\bf \Sigma}, \mu)$ be a complete probability space. Suppose $P$ is a contractive projection on $L^p(\Omega, {\bf \Sigma}, \mu)$ satisfying $P{\mathbf 1}={\mathbf 1}$, then there exists a complete sub $\sigma$-algebra $\mathbf{\Sigma}'$ of ${\bf \Sigma}$ such that $P=\E(\cdot |\mathbf{\Sigma}')$.
	\end{Theorem A}

\noindent Based on Theorem A and a transfer argument, they completely solve the Contractive Projection Problem for $L^p(\Omega, {\bf \Sigma}, \mu)$ (the case $p=1$ is due to Douglas and the case $p\neq 1$ is due to And\^o); see \cite[Theorem 2]{An}. We note that, in addition to Theorem A, there have been several further studies on operator-theoretic characterizations of conditional expectations; see \cite{Ran05, Ran, DHP90}.

In contrast with the case $p \geq 1$, And\^o proved that, for $0<p<1$, only very few contractive projections exist on $L^p(\mu)$.

\begin{Theorem B}[And\^o's theorem] 
	Let $0<p<1$. If $P$ is a contractive projection on $L^p$ satisfying $P\one =\one$, then $P$ is the identity.
\end{Theorem B}

Furthermore, by applying the transfer argument, And\^o proved that any contractive projection $P$ on $L^p(\mu)$ ($0<p<1$) is precisely of the form $$P=M_{\chi_B} + V,$$ where $B\in {\bf \Sigma}$ is of positive measure, $M_{\chi_B}$ is the multiplication operator with symbol $\chi_B$, and $V$ is an isometry on $L^p(\mu)$ satisfying $$M_{\chi_B}V= V, \quad VM_{\chi_B}=0.$$ 

\subsection{Main results}
As mentioned earlier, we wish to characterize contractive projections on the Hardy space $H^p(\T)$, which is a closed subspace of $L^p(\T)$. We first establish the following more general theorem.

	\begin{theorem}\label{1.1}
	Let $0<p<1$. Let $X$ be a closed subspace of $L^p(\mu)$ containing the constants and let $P:X\to X$ be a contractive projection satisfying $P{\bf 1}={\bf 1}$. Then there exists a complete sub $\sigma$-algebra $\bf \Sigma'$ of ${\bf \Sigma}$ such that 
	\begin{align}\label{eq1.2}
		Pf=\E(f| {\bf \Sigma'}), \quad \forall f\in \mathcal E,
	\end{align}
	where $\mathcal E:= \big\{f\in X: f\in L^\infty(\mu), Pf\in L^\infty(\mu) \big\}.$
	\end{theorem}
	\noindent The proof of Theorem \ref{1.1} will be presented in Section 2. We note that the space $\mathcal E$ is nonzero, since $\one\in \mathcal E$, but for certain choices of $X$ and $P$, it may consist solely of constant functions. In such cases, Theorem \ref{1.1} does not yield any meaningful information. But Theorem \ref{1.1} applies well to our study of contractive projections on the Hardy spaces. 
	
	We recall that for $0<p<\infty$, the Hardy space $H^p(\T)$ consists of all functions in $L^p(\T)$ that can be approximated by analytic polynomials in the $L^p$-norm. By the boundedness of evaluation functionals on $\mathbb{D}$, every $f\in H^p(\T)$ can be naturally identified with an analytic function on $\mathbb{D}$ whose radial limit is equal to $f$ almost everywhere on $\T$ \cite{Du70}.

Let $0<p<1$, and let $P$ be a contractive projection on $H^p(\mathbb{T})$ such that $P{\bf 1}={\bf 1}$. We assume $P z^{n}$ is bounded for each $n\geq 1$, then it follows from Theorem \ref{1.1} that $P=\E(\cdot |{\bf \Sigma}')$ on $H^p(\T)\cap L^1(\T)=H^1(\T)$. Thus this conditional expectation $\E(\cdot |{\bf \Sigma}')$ must leave $H^1(\T)$ invariant. According to an elegant theorem of Aleksandrov \cite{Al}, any such ${\bf \Sigma}'$ is either trivial or generated by an inner function $\eta$ vanishing at the origin.

Indeed, if $p\geq 1$, then for any inner functions $\eta$ vanishing at the origin, the conditional expectation $\mathbb{E}(\cdot|\eta)$ is a contractive projection on $H^p(\T)$. However, for $0<p<1$, $\mathbb{E}(f |\eta)$ is well defined only for integrable functions $f$, and this densely defined operator may even be unbounded on $H^p(\T)$. The following theorem, proved in Section 3, examines when $\mathbb{E}(\cdot| \eta)$ can be extended boundedly to $H^p(\T)$ for $0<p<1$, and evaluates its norm.
	\begin{theorem}\label{1.2}
	 For $0<p<1$, the conditional expectation $\mathbb{E}(\cdot|\eta)$ is bounded on $H^p(\T)$ if and only if $\eta$ is a finite Blaschke product. Furthermore, for a finite Blashcke product $\eta$ with $\eta(0)=0$, we have
	\begin{equation*}
			\|\mathbb{E}(\cdot|\eta)\|_{H^p}=\|\mathbb{E}(\cdot|\eta)\|_{L^p}=\|\eta'\|_{\infty}^{\frac{1}{p}-1}.
		\end{equation*}
	\end{theorem}

	\noindent The proof of Theorem \ref{1.2} is based on a pointwise representation of conditional expectations. Note that if $\eta$ is a finite Blaschke product with $\eta(0)=0$, then $\|\eta'\|_\infty = 1$ if and only if $\eta=cz$ for some unimodular constant $c$. As a consequence of Theorem \ref{1.1} and \ref{1.2}, there is no nontrivial contractive projection on $H^p(\T)$ satisfying $P\one=\one$ that maps each monomial to a bounded function.

	\begin{theorem}\label{1.3}
	Suppose $0<p<1$. Let $P: H^p(\T)\to H^p(\T)$ be a contractive projection satisfying $P{\bf 1}={\bf 1}$. If $P z^n\in H^\infty(\mathbb{T})$ for each $n\geq 1$, then $$Pf=f(0),\quad f\in H^p(\T),$$ or $P$ is the identity.
  \end{theorem}

Note that projections induced by {\it idempotent coefficient multipliers} map monomials to bounded functions. In Section 4, we will apply Theorem \ref{1.3} to characterize contractive idempotent coefficient multipliers on $H^p(\T^d)$.

	\begin{theorem}\label{1.4}
	Suppose that $0<p<1$ and $1\leq d\leq \infty$. Then $P_\Gamma$ defined by \eqref{eq1.1} extends to a contraction on $H^p(\T^d)$ if and only if there exists a subset $J \subseteq \{1,2,\ldots ,d\}$ $( J \subseteq \{1,2,\ldots \}$ if $d=\infty)$ , such that $$\Gamma=\{\alpha\in \mathbb N_0^{(d)}: \alpha_j =0\,\, {\rm for}\,\, j\in J\}.$$
	\end{theorem}
\noindent This complements a remarkable result of Brevig, Ortega-Cerd\`{a} and Seip \cite{BOS} which characterizes such multipliers on $H^p(\T^d)$ for $1\leq p \leq \infty$ and $1\leq d\leq \infty$. In the case $d=\infty$, Theorem \ref{1.4} also admits a routine transfer to the setting of Dirichlet series. 
	
	\vspace{.2cm}
\noindent{\bf Notations.} In this paper, $(\Omega,  \bf \Sigma, \mu)$ always denotes a complete probability space. The completeness means $\bf \Sigma$ contains all $\mu$-null sets. For a family $\cal G$ of measurable functions on $(\Omega,  \bf \Sigma, \mu)$, $\bf \Sigma(\cal G)$ denotes the smallest complete sub $\sigma$-algebra that makes every function in $\cal G$ measurable. For simplicity, if $\cal G=\{g_1,g_2,\ldots \}$ is countable, we write ${\bf \Sigma}(g_1,g_2,\ldots)$ instead of ${\bf \Sigma}(\{g_1,g_2, \ldots\})$. For a $\bf \Sigma$-measurable function $\eta$, we write $$\E(\cdot | \eta):= \E(\cdot| {\bf \Sigma}(\eta)).$$
		
	\section{An extension theorem for the case $0<p<1$}

	This section is devoted to the proof of Theorem \ref{1.1}. We begin with the following lemma.

	\begin{lemma}\label{2.1}
	Let $0<p<1$. Let $X$ be a closed subspace of $L^p(\mu)$ and $P:X\to X$ be a contractive projection. Suppose $f\in {\rm Ker}\, P$ and $g\in \Ran P$. If $f, g$, and $1/g$ are all in $L^\infty(\mu)$, then we have $$\int_\Omega |g|^{p-2} g \bar{f} d\mu=0.$$
	\end{lemma}
	In the case $1\leq p<\infty$, similar integral identities can be proved without any additional assumptions on the boundedness of $f$, $g$, and $1/g$, which provide effective criteria for Birkhoff-James orthogonality in $L^p$ spaces ($1\leq p<\infty$); see \cite{Sha} for details. This criterion fails when $0<p<1$. However, under the assumption that $f$, $g$, and $1/g$ are bounded, Lemma \ref{2.1} can be established by a variational argument. The proof is straightforward and is omitted here.

	We also need the following two elementary observations concerning the $\sigma$-algebra generated by a family of functions. These lemmas may be well known to experts, but we include a proof here for completeness.

	\begin{lemma}\label{2.2}
	Let $\cal G$ be a family of measurable functions on $(\Omega, \bf \Sigma, \mu)$. Denote by $\bf \Sigma(\cal G)$ the $\sigma$-algebra generated by $\cal G$. For any $A\in \bf \Sigma(\cal G)$, there exist countably many functions $\{g_i\}_{i=1}^{\infty}$ in $\cal G$, such that $A\in {\bf \Sigma}(g_1,g_2, \ldots )$.
	\end{lemma}
	\begin{proof} 
	Consider the union
	$$\mathscr F:=\bigcup_{\cal C \subseteq \cal G, \atop {\rm \cal C\,\,is\,\, countable}} \bf \Sigma(\cal C).$$ The result follows once if we show $\mathscr F=\bf \Sigma(\cal G)$. Since $\mathscr F\subseteq \bf\Sigma(\cal G)$, it is enough to show $\mathscr F$ is a complete $\sigma$-algebra. Clearly, $\mathscr F$ is closed under complement and contains all null sets. If $A\in \bf\Sigma(\cal C)$ and $B\in \bf\Sigma(\cal C')$, then $A\cap B \in \bf\Sigma(\cal C\cup \cal C')\subseteq \mathscr F$. Similarly, if $A_n\in {\bf\Sigma}(\cal{C}_n)$ for $n\geq 1$, then $\cup_{n\geq 1} A_n \in {\bf \Sigma}(\cup_{n\geq 1}\cal C_n)$. Thus $\mathscr F$ is also closed under intersection and countable unions. This completes the proof.
	\end{proof}
	
	\begin{lemma}\label{2.3}
		Let $\cal G$ be a subset of $L^\infty$ and $f\in L^1$. If $\E\big(f|{\bf \Sigma}(g_1,\ldots,g_n)\big)=0$ for any finite subset $\{g_1,\ldots,g_n\} \subseteq \cal G$, then
		$$
		\E\big(f|{\bf \Sigma}(\cal G)\big)=0.
		$$
	\end{lemma}
	\begin{proof}
		For any $A\in{\bf \Sigma}(\cal G)$, Lemma \ref{2.2} ensures that there exist at most countably many functions $\{g_i\}_{i=1}^\infty$ in $\cal G$ such that $A\in{\bf \Sigma}(g_1,g_2,\ldots)$. We deduce from
		$$
		\E(f|{\bf \Sigma}(g_1,\ldots,g_n))=0, \quad n\geq 1,
		$$
		and Doob's martingale convergence theorem that
		$$
		\E(f|{\bf \Sigma}(g_1,g_2,\ldots))=0.
		$$
		This gives $$\int_A f d\mu=0.$$ Since $A\in {\bf \Sigma}(\cal G)$ is arbitrary, we conclude that $\E(f|{\bf \Sigma}(\cal G))=0$.
	\end{proof}

		Now we are ready to prove Theorem \ref{1.1}.
		\begin{proof}[Proof of Theorem \ref{1.1}]
		Let ${\bf \Sigma}'={\bf \Sigma} (P(\mathcal {E}))$. We claim that for each $f\in \mathcal E$ it holds that $$Pf=\mathbb E(f| \bf \Sigma').$$ 
		By Lemma \ref{2.3}, it is enough to prove for any finite subset $\{g_1,\ldots,g_n\}\subseteq P(\cal E)$,
		$$
		\E(f-Pf|{\bf \Sigma}(g_1,\ldots, g_n))=0,\quad f\in \mathcal E.
		$$
		Since $g_k\in L^\infty$ for each $1\leq k\leq n,$ there exists a $\delta>0$ such that ${\bf 1}+\sum_{k=1}^n w_k g_k$ is invertible in $L^\infty$ for any $w=(w_1,\ldots,w_n)$ with $|w_k|\leq \delta,~1\leq k\leq n$. It follows from Lemma \ref{2.1} that
		\begin{align}\label{eq2.1}
			\int_\Omega |1+\sum_{j=1}^n w_j g_j|^{p-2}(1+\sum_{j=1}^n w_j g_j)(\ol{f-Pf}) d\mu=0.
		\end{align}
		 Now we consider the following infinite series expansion
		\begin{eqnarray}
			\big|1+\sum_{j=1}^n w_j g_j\big|^{p-2}(1+\sum_{j=1}^n w_j g_j)
			&=&\big(1+\sum_{j=1}^n w_j g_j\big)^{p/2}\big(1+\sum_{j=1}^n {\bar w_j} \, {\bar g_j}\big)^{p/2-1}\nonumber\\
			&=&\sum_{\alpha\in \bb N_0^n} \sum_{\beta\in \bb N_0^n} \tbinom{p/2}{\alpha}\tbinom{p/2-1}{\beta} g^\alpha\bar{g}^\beta w^\alpha \bar {w}^\beta. \,\label{eq2.2}
		\end{eqnarray}
		Here, $$g:\Omega\rightarrow \C^n; \quad x\mapsto\big(g_1(x),\ldots,g_n(x)\big),$$
        and 
		$$\tbinom{x}{\alpha}= \frac{\Gamma(x+1)}{\Gamma(\alpha_1+1)\cdots \Gamma(\alpha_n+1)\Gamma(x-|\alpha|+1)}, \quad x\notin \mathbb Z,$$ where $|\alpha|:= \alpha_1+ \cdots +\alpha_n$.
		Note that $\delta$ can be chosen sufficiently small to ensure the series \eqref{eq2.2} converges uniformly on $w\in \delta \overline {\mathbb D}^n$ for almost all $x\in \Omega$. Multiplying $\ol{f-Pf}$ in two sides of \eqref{eq2.2} and then integrating with respect to $x\in \Omega$, it follows from Equation \eqref{eq2.1} that
		\begin{align}\label{eq2.3}
			\int_{\Omega} g^\alpha \bar{g}^\beta \,(\ol{f-Pf})d\mu=0, \quad \forall \,\alpha,\beta \in \mathbb N^n_0.
		\end{align}
		Consider the pushforward measure $d\nu:=g_*\Big((\ol{f-Pf})d\mu\Big)$ on $\C^n$ defined for each Borel set $B\subseteq \C^n$ as 
		$$
		\nu(B)=\int_{g^{-1}(B)}(\ol{f-Pf})d\mu.
		$$ 
		Since the functions $g_k$ are all in $L^\infty(\mu)$, the measure $\nu$ is supported on a compact subset $K\subseteq\C^n$. By \eqref{eq2.3} and the change of variables formula for the pushforward measure, we obtain
		\begin{equation}\label{eq2.4}
			\int_{K}z^\alpha\bar{z}^\beta d\nu(z)=0,\quad \forall \,\alpha,\beta \in \mathbb N^n_0.
		\end{equation}
		Since $\{z^\alpha\bar{z}^\beta: \alpha,\beta \in \mathbb N^n_0\}$ spans a dense subsapce of $C(K)$, the identities in \eqref{eq2.4} imply $\nu=0$. This means for any Borel set $B\subseteq\C^n$, 
		\begin{equation}\label{eq2.5}
			\int_{g^{-1}(B)}(\ol{f-Pf})d\mu=\int_B d\nu=0.
		\end{equation}
		Note that any member $A\in {\bf \Sigma}(g_1,\ldots, g_n)$ is exactly of the form $$A= g^{-1}(B) \Delta N$$ where $B$ is a Borel set in $\C^n$ and $N\subseteq \Omega$ is a $\mu$-null set.
		We conclude from \eqref{eq2.5} that
		\begin{equation*}
			\int_A (\ol{f-Pf})d\mu=0,\quad \forall \,A\in{\bf \Sigma}(g_1,\ldots,g_n),
		\end{equation*}
		and hence
		\begin{equation*}
			\E(f-Pf|{\bf \Sigma}(g_1,\ldots,g_n))=0.
		\end{equation*}
		The proof is completed.
	\end{proof}

	\section{conditional expectation operators on $H^p(\mathbb{T}),~0<p<1$.}

As mentioned in the introduction, Theorem \ref{1.1} can be applied to address the problem of when a projection on $H^p(\T)$ with $0<p<1$ can be realized by a conditional expectation. In this section, we consider the following inverse problem and investigate which conditional expectations yield contractive projections on $H^p(\T)$. 

Note that conditional expectations are defined only for integrable functions, rather than for every function in $L^p(\T)$ with $0<p<1$. Therefore, before considering restrictions of conditional expectations to $H^p(\T)$, it may be helpful to first understand how they act on $L^p(\T)$. In particular, the boundedness of conditional expectations on $L^p(\T)$ for $0<p<1$ was fully characterized by N. J. Kalton \cite[Theorem 4.4]{Kal78}. Indeed, for these values of $p$, not every conditional expectation is bounded on $L^p(\T)$, and by And\^o's theorem (Theorem B), the only contractive one is the identity.

However, in order to obtain a contractive projection by restricting some $\E(\cdot|{\bf \Sigma}')$ to $H^p(\T)$ with $0<p<1$, the desired conditional expectation must satisfy the following conditions in a successive manner: 
\begin{itemize}
	\item[$(\mathcal F 1)$:] For any $f\in H^1(\T)$, $\E(f| {\bf \Sigma}') \in H^1(\T)$;
	\item[$(\mathcal F 2)$:] As an operator densely defined on $H^1(\T)$, $\E(\cdot|{\bf \Sigma}')$ is contractive with respect to the $p$-norm. Specifically, $$\|E(f|{\bf \Sigma}')\|_p \leq \|f\|_p,\quad \forall f\in H^1(\T).$$ 
\end{itemize}

The following elegant theorem of A. B. Aleksandrov \cite{Al} completely characterizes the sub $\sigma$-algebras ${\bf \Sigma}'$ that satisfy the condition $(\mathcal{F}1)$ above.

\begin{theorem}[Aleksanderov's Theorem]\label{3.1}
	The conditional expectation $\E(\cdot| {\bf \Sigma}')$ on $L^1(\T)$ leaves $H^1(\T)$ invariant if and only if either ${\bf \Sigma}'$ is generated by the constant function ${\bf 1}$ or ${\bf \Sigma}'$ is generated by an inner function with $\eta(0)=0$.
\end{theorem}

\noindent The `if' part of Aleksanderov's theorem is easy. Indeed, if $\eta$ is an inner function vanishing at the origin, then $\{\eta^k: k\in \mathbb Z\}$ forms an orthonormal basis of $L^2\big(\T,{\bf \Sigma}(\eta),dm\big)$. Since $\E(\cdot| \eta)$ is the orthogonal projection from $L^2(\T,dm)$ onto $L^2\big(\T, {\bf \Sigma}(\eta),dm\big)$, we have
$$\E(f| \eta)=\sum_{k\in \mathbb Z} \langle f, \eta^k \rangle \,\eta^k$$
for any $f\in L^2(\T,dm)$, and hence $\E(\cdot| \eta)$ leaves $H^2(\T)$ invariant. An approximation argument shows that $\E(\cdot| \eta)$ leaves $H^1(\T)$ invariant.

In view of Aleksandrov's theorem, we only need to consider for which inner functions $\eta$ with $\eta(0)=0$, the $\sigma$-algebra generated by $\eta$ satisfies condition $(\mathcal{F}2)$ above. In what follows, we will show, as stated in Theorem \ref{1.2}, that when $0<p<1$, the conditional expectation $\mathbb{E}(\cdot|\eta)$ is bounded on $H^p$ if and only if $\eta$ is a finite Blaschke product. Furthermore, if $\eta$ is a finite Blaschke product, then $$\|\mathbb{E}(\cdot|\eta)\|_{H^p} = \sup_{w\in \T} |\eta'(w)|^{\frac{1}{p}-1}.$$ 

The next lemma is a part of the Duren-Romberg-Shields theorem \cite[Theorem 1]{DRS} characterizing the dual space of $H^p(\T)$ for $0<p<1$. 

\begin{lemma}\label{3.2}
Let $0<p<1$ and let $\lambda$ be a continuous linear functional on $H^p(\T)$. Then there exists a unique function $g\in A(\D)$ such that
\begin{align*}
\lambda(f)= \lim_{r\to 1} \frac{1}{2\pi}\int_0^{2\pi} f(re^{i\theta}) \ol{g(e^{i\theta})} d\theta, \quad f\in H^p(\T).
\end{align*}
Here, $A(\D)$ is the classical disc algebra, consisting of all continuous functions on $\ol{\D}$ that are analytic in $\D$.
\end{lemma}

\begin{theorem}\label{3.3}
	Let $0<p<1$ and let $\eta$ be an inner function with $\eta(0)=0$. If $\mathbb{E}(\cdot|\eta)$ is bounded on $H^p(\mathbb{T})$, then $\eta$ is a finite Blaschke product.
\end{theorem}

\begin{proof}
	Write $\eta(z)=z^k \xi(z)$ where $\xi$ is an inner function and $\xi(0)\neq 0$. Denote by $\tau_k$ the functional
    $$\tau_k: f\mapsto\widehat{f}(k ),$$ where $\widehat{f}(k)$ is the $k$-th Taylor coefficient of $f$. Because both $\tau_k $ and $\E(\cdot| \eta)$ are bounded on $H^p(\T)$, we see that $$\lambda:=\tau_k\circ \mathbb{E}(\cdot|\eta): f \mapsto \lim_{r\to 1}\frac{1}{2\pi}\int_0^{2\pi} \E(f|\eta)(re^{i\theta})\cdot \ol{e^{i k \theta}} d\theta$$ defines a continuous linear functional on $H^p(\mathbb{T})$. By Lemma \ref{3.2}, there exists a unique function $g\in A(\D)$ such that $$\lambda(f)=\lim_{r\to 1}\, \frac{1}{2\pi}\int_0^{2\pi} f(re^{i\theta}) \ol{g(e^{i\theta})} d\theta,\quad f\in H^p(\T).$$ In particular, for any analytic polynomial $q$, it holds that  $$\lambda(q)=\int_\T q(w) \ol{g(w)} dm(w).$$
	Since $\E(q| \eta)\in H^\infty(\T)$ for any analytic polynomial $q$ and $\E(\cdot| \eta)$ is self-adjoint as an operator on $L^2(\T,dm)$, we have
	\begin{align*}
		\lambda(q)&=\int_\T \E(q|\eta)(w)\cdot \ol{w}^k dm(w)\\
		&=\int_\T q(w)\cdot  \ol{\E(z^k| \eta)(w)} \,dm(w)\\
		&=\int_\T q(w)\cdot  \ol{\bigg(\sum_{l \geq 0} \langle z^k , \eta^l \rangle \eta^l(w) \bigg)}dm(w).\\	
		&=\int_\T q(w) \cdot \xi(0) \ol{\eta(w)} \,dm(w)
	\end{align*}	
	As a result, the identity $$\int_\T q(w) \ol{g(w)} dm(w)=\int_\T q(w) \cdot \xi(0) \ol{\eta(w)} dm(w)$$ holds for any analytic polynomial $q$. We conclude that $g=\ol{\xi(0)} \eta$. This shows $\eta$ is an inner function that lies in $A(\mathbb{D})$, and hence it must be a finite Blaschke product.
\end{proof}

Next, we consider the converse of Theorem \ref{3.3}. We first recall some basic properties for finite Blaschke products. Let $\eta(z)=\prod_{j=1}^k\frac{z-a_j}{1-\overline{a_j}z}$ be a finite Blaschke product where $a_j\in\mathbb{D},~1\leq j\leq k$. A direct calculation yields
\begin{equation*}
	\sum_{j=1}^k\frac{1-|a_j|}{1+|a_j|}\leq |\eta'(z)|\leq \sum_{j=1}^k\frac{1+|a_j|}{1-|a_j|},\quad z\in\mathbb{T}.
\end{equation*}
This shows that $\eta$ is a locally univalent function on $\T$. 
Actually, $\eta$ is a covering map from $\T$ onto itself with degree $k$. To be more specific, for each $z\in \T$ the preimage $\eta^{-1}(z)$ consists of exactly $k$ distinct points on $\T$. And if we write $\eta^{-1}(z)=\{w_m: 1\leq m \leq k\}$ with $${\rm Arg}\, w_1 < {\rm Arg}\,  w_2<\cdots < {\rm Arg}\, w_k < 2\pi+{\rm Arg}\, w_1,$$ then $\eta$ maps each arc $[w_m, w_{m+1})$ to $\T$ bijectively ($w_{k+1}:=w_1$).

The following identity plays a crucial role in the remainder of the proof.

\begin{lemma}\label{3.4}
	Let $\eta$ be a finite Blaschke product. Then for any $f\in L^1(\mathbb{T})$,
	\begin{equation*}
		\int_{\mathbb{T} }f(z)|\eta'(z)| dm(z)=\int_{\mathbb{T}} \bigg(\sum_{w\in \eta^{-1}(z)}f(w) \bigg)dm(z).
	\end{equation*}
\end{lemma}

\begin{proof}
	By an approximation argument, it is enough to establish the desired identity when $f=\chi_I$ for a sufficiently small arc $I\subseteq \T$. Since $\eta'(z)\neq 0$ for all $z\in\T$, and hence $\eta$ is locally univalent. We may assume the arc $I$ is sufficiently small so that $\eta$ is injective on $I$. With this assumption, $$\sum_{w\in \eta^{-1}(z)} \chi_I(w) = \sum_{w\in \eta^{-1}(z)\cap I} 1 = \chi_{\eta(I)}(z),$$ and hence
	\begin{align*} 
		\int_{\T} \sum_{w\in \eta^{-1}(z)} \chi_I(w) dm(z) = \int_{\T} \chi_{\eta(I)}(z) dm(z)=m(\eta(I))=\int_\T \chi_I(z) |\eta'(z)|dm(z).
	\end{align*}
	This completes the proof.
\end{proof}

When we fix a finite Blaschke product $\eta$, the notation $\thicksim $ stands for the equivalence induced by $\eta$. Namely, for two points $z,w\in \T$, $w\thicksim z$ if $\eta(w)=\eta(z).$

\begin{proposition}\label{3.5}
	Let $\eta$ be a finite Blaschke product with $\eta(0)=0$. Then
	\begin{equation*}
		\sum_{w\thicksim z}\frac{1}{|\eta'(w)|}=1, \quad {\rm for \,\, each \,\,} z\in \T.
	\end{equation*}
	In particular, $\|\eta'\|_{\infty}\geq k$, where $k$ is the degree of $\eta$.
\end{proposition}

\begin{proof}
	We write  $F(z)=\sum_{w\thicksim z}\frac{1}{|\eta'(w)|}$ and $G(z)=\sum_{w\in \eta^{-1}(z){}}\frac{1}{|\eta'(w)|}$. Clearly, $F(z)=G(\eta(z))$.
	For each $n\in\mathbb{Z}$, applying Lemma \ref{3.4} to $f=\eta^n/|\eta'|$ we obtain
	\begin{align*}
	\int_{\mathbb{T}}\sum_{w\in\eta^{-1}(z)}\frac{\eta^n(w)}{|\eta'(w)|}dm(z)=\int_\mathbb{T}\eta^n(z)dm(z)= \begin{cases}
	1   &  n=0;\\
         0 &  n\neq 0.
	\end{cases}
	\end{align*}
	This implies that
	\begin{align}\label{eq3.1}
	\int_{\mathbb{T}} z^n G(z)dm(z)=\begin{cases}
	1   &  n=0;\\
         0 &  n\neq 0.
	\end{cases}
	\end{align}
	Note that $G$ is continuous. So \eqref{eq3.1} implies $G(z)=1$ for all $z\in \T$, and hence $F(z)=G(\eta(z))=1$ for all $z\in \T$.
\end{proof}

\begin{proposition}\label{3.6}
	Let $f$ be a Lebesgue measurable function on $\mathbb{T}$ and let $\eta$ be a finite Blaschke product. Then the function $F(z)=\sum_{w\thicksim z} f(w)$ is $\mathbf{\Sigma}(\eta)$-measurable.
\end{proposition}

\begin{proof}
	It suffices to prove the case when $f=\chi_E$ is an indicator function for any Lebesgue measurable set $E$. Writing $\mathbb{T}$ as a finite union of arcs on which $\eta$ is injective, we may assume, without loss of generality, that $\eta$ is injective on $E$. Then $F=\chi_{\eta^{-1}(\eta(E))}$, and we need to show that $\eta^{-1}\big(\eta(E)\big)$ is $\mathbf{\Sigma}(\eta)$-measurable. Since $\eta$ is a Lipschitz function on $\mathbb{T}$, $\eta(E)$ is Lebesgue measurable. Thus $\eta(E)=F_0\cup F_1$, where $F_0$ is a null set and $F_1$ is a Borel set. Note that $\eta^{-1}(F_0)$ is also null. This implies $\eta^{-1}\big(\eta(E)\big)=\eta^{-1}(F_0)\cup\eta^{-1}(F_1)$ is $\mathbf{\Sigma}(\eta)$-measurable.
\end{proof}

With the above preparations, we are able to give a pointwise expression for the conditional expectation. 	
\begin{proposition}\label{3.7}
	Let $\eta$ be a finite Blaschke product with $\eta(0)=0$. Then for any $f\in L^1(\mathbb{T})$,
	$$
	\mathbb{E}(f|\eta)(z)=\sum_{w\thicksim z}\frac{f(w)}{|\eta'(w)|}.
	$$
\end{proposition}

\begin{proof}
	Given a function $f \in L^1(\mathbb{T})$, the function $$F(z)=\sum_{w \thicksim z} \frac{f(w)}{|\eta'(w)|}$$ is $\mathbf{\Sigma}(\eta)$-measurable by Proposition \ref{3.6}. As a consequence, to show $\E(f| \eta)=F$ it is enough to establish
	\begin{equation*}
		\int_{\mathbb{T}} \eta^n(z) F(z) \, dm(z)=\int_{\mathbb{T}} \eta^n(z) f(z) \, dm(z),
	\end{equation*}
	for each $n \in \mathbb{Z}$.
	This can be deduced, by applying Lemma \ref{3.4}, Proposition \ref{3.5}, and the following calculations:
	\begin{align*}
		&\int_{\mathbb{T}} \eta^n(z) F(z) \, dm(z)\\
		=&\int_{\mathbb{T}} \bigg(\sum_{w \in \eta^{-1}(z)} \frac{\eta^n(w) \cdot F(w) }{|\eta'(w)|}\bigg)\,dm(z) \\
		=&\int_{\mathbb{T}} z^n \cdot \bigg(\sum_{w \in \eta^{-1}(z)} \frac{F(w)}{|\eta'(w)|} \bigg)\, dm(z)\\
		=&\int_\T z^n \cdot \bigg( \sum_{w \in \eta^{-1}(z)} \frac{1}{|\eta'(w)|} \sum_{\zeta \thicksim w}\frac{f(\zeta)}{|\eta'(\zeta)|}\bigg) \, dm(z)\\
		=&\int_\T z^n \cdot \bigg(\sum_{w \in \eta^{-1}(z)} \frac{1}{|\eta'(w)|}\bigg) \cdot \bigg( \sum_{\zeta \in \eta^{-1}(z)} \frac{f(\zeta)}{|\eta'(\zeta)|}\bigg) \, dm(z)\\
		=&\int_\T z^n \cdot \bigg( \sum_{\zeta \in \eta^{-1}(z)} \frac{f(\zeta)}{|\eta'(\zeta)|}\bigg)\, dm(z)\\
		=&\int_{\mathbb{T}} \sum_{\zeta \in \eta^{-1}(z)} \frac{\eta^n(\zeta) \cdot f(\zeta) }{|\eta'(\zeta)|}\, dm(z) \\
		=&\int_{\mathbb{T}} \eta^n(z) f(z) \, dm(z).
	\end{align*}
	The proof is completed.
	\end{proof}

We now proceed to prove the remaining part of Theorem \ref{1.2}. 

\begin{proof}[Proof of Theorem \ref{1.2}]

	It remains to prove
	\begin{equation*}
		\|\mathbb{E}(\cdot|\eta)\|_{H^p}=\|\mathbb{E}(\cdot|\eta)\|_{L^p}=\|\eta'\|_{\infty}^{\frac{1}{p}-1}
	\end{equation*}
	for any finite Blaschke product $\eta$ with $\eta(0)=0$.

	Clearly, $\|\mathbb{E}(\cdot|\eta)\|_{H^p}\leq\|\mathbb{E}(\cdot|\eta)\|_{L^p}$. We first establish 
    $\|\mathbb{E}(\cdot|\eta)\|_{L^p}\leq \|\eta'\|_\infty^{\frac{1}{p}-1}$. For any $f\in L^1(\mathbb{T})$, by proposition \ref{3.7} we have
	\begin{align*}
		\int_{\mathbb{T}}|\mathbb{E}(f|\eta)(z)|^p \, dm(z) &= \int_{\mathbb{T}} \,\,\left| \sum_{w\sim z} \frac{f(w)}{|\eta'(w)|} \right|^p dm(z) \\
		&\leq \int_{\mathbb{T}} \,\,\sum_{w\sim z} \frac{|f(w)|^p}{|\eta'(w)|^p} dm(z) \\
		&= \int_{\mathbb{T}} \mathbb{E}\Big(|f|^p|\eta'|^{1-p}\Big|\eta\Big)(z) \, dm(z) \\
		&= \int_{\mathbb{T}} |f(z)|^p |\eta'(z)|^{1-p} \, dm(z) \\
		&\leq \|f\|^p_p \|\eta'\|_\infty^{1-p}.
	\end{align*}
	Thus, $\|\mathbb{E}(\cdot|\eta)\|_{L^p}\leq \|\eta'\|_\infty^{\frac{1}{p}-1}$.
	
	Next, we show $\|\mathbb{E}(\cdot|\eta)\|_{H^p}\geq \|\eta'\|_\infty^{\frac{1}{p}-1}$. Given $z_0\in\mathbb{T}$ and $\varepsilon>0$, there exists an open arc $I$ containing $z_0$ such that for each $z\in I$,
	\begin{equation*}
		\Big| \frac{1}{|\eta'(z_0)|}-\frac{1}{|\eta'(z)|} \Big| <\varepsilon,
	\,\,\,{\rm and }\,\,\,
		\Big| |\eta'(z_0)|-|\eta'(z)| \Big| <\varepsilon.
	\end{equation*}
	Moreover, we can choose $I$ sufficiently small so that $J:= \eta^{-1}\big(\eta(I)\big)$ consists of exactly $k$ disjoint open arcs.
	Now for sufficiently large $t$, there exists a function $f_t\in H^\infty(\mathbb{T})$ such that
	\begin{equation*}
		|f_t(z)|=
		\begin{cases}
			t &  \quad z\in I; \\
			1 & \quad z\notin I.
		\end{cases}
		\quad
	\end{equation*}
	It follows from Proposition \ref{3.7} that
	\begin{align}\label{eq3.2}
		|\mathbb{E}(f_t|\eta)(z)|&= \bigg| \sum_{w\thicksim z}\frac{f_t(w)}{|\eta'(w)|}\bigg|\nonumber\\
		 &\geq t\left(\frac{1}{|\eta'(z_0)|}-\varepsilon\right)-(k-1)\big\|\frac{1}{\eta'}\big \|_\infty, \quad z\in J.
	\end{align}
	and
	\begin{align}\label{eq3.3}
	m(J)&= \int_{\T} \,\, \sum_{w\thicksim z} \chi_I(w)\, dm(z)\nonumber\\
	&=\int_\T \E\Big(|\eta' | \cdot \chi_I \Big |\eta\Big)(z) dm(z)\nonumber\\
	&=\int_\T |\eta' (z)| \cdot \chi_I(z) dm(z)\nonumber\\
	&\geq \big(|\eta'(z_0)|-\varepsilon\big) \,m(I).
	\end{align}
	Combining the lower estimations \eqref{eq3.2} and \eqref{eq3.3} we have
	\begin{align}\label{eq3.4}
		\|\E(f_t| \eta)\|_{p} \geq  \left(t\left(\frac{1}{|\eta'(z_0)|}-\varepsilon\right)-(k-1)\big\|\frac{1}{\eta'}\big \|_\infty \right)\cdot \big(|\eta'(z_0)|-\varepsilon\big)^{1/p} \,m(I)^{1/p}.
	\end{align}
	Moreover, we have an upper bound
	\begin{equation}\label{eq3.5}
		\|f_t\|_{p} \leq \big( t^p m(I)+1\big)^{1/p}.
	\end{equation}
	Now by \eqref{eq3.4} and \eqref{eq3.5}
	\begin{align*}
		\|\mathbb{E}(\cdot|\eta)\|_{H^p}\geq &\limsup _{t\to \infty} \frac{\|\E(f_t| \eta)\|_p }{\|f_t\|_{p}}\\
		\geq &\left(\frac{1}{|\eta'(z_0)|}-\varepsilon\right)\big(|\eta'(z_0)|-\varepsilon\big)^{\frac{1}{p}}.
	\end{align*}
	Letting $\varepsilon\rightarrow 0$, we see $\|\mathbb{E}(\cdot|\eta)\|_{H^p}\geq |\eta'(z_0)|^{\frac{1}{p} -1}$. Since $z_0\in \T$ is arbitrary,
	\begin{equation*}
		\|\mathbb{E}(\cdot|\eta)\|_{H^p}\geq \|\eta'\|_\infty^{\frac{1}{p}-1} .
	\end{equation*}
	This finishes the proof of Theorem \ref{1.2}. 
	\end{proof}

	For convenience, we restate Theorem \ref{1.3} below. It is an immediate consequence of Theorems \ref{1.1} and \ref{1.2}.

	\begin{theorem}\label{3.8}
	Suppose $0<p<1$. Let $P: H^p(\T)\to H^p(\T)$ be a contractive projection satisfying $P{\bf 1}={\bf 1}$. If $P z^n\in H^\infty(\mathbb{T})$ for each $n\geq 1$, then either $$Pf=f(0),\quad f\in H^p(\T),$$ or $P$ is the identity.
  \end{theorem}

	\begin{proof} 
	 By Theorem \ref{1.1}, there exists a sub $\sigma$-algebra ${\bf \Sigma}'$ such that $$Pf=\E(f|{\bf \Sigma}'), \quad \forall f\in \mathcal E,$$ 
     where $\cal E=\{f\in H^\infty(\T): Pf\in H^\infty(\T) \}$.
     Since $P$ maps monomials to bounded functions, $\cal E$ contains all analytic polynomials. This implies that $\E(\cdot |{\bf \Sigma}')$ leaves $H^1(\T)$ invariant. Then it follows from Aleksanderov's theorem that ${\bf \Sigma}'$ is trivial or generated by an inner function $\eta$ with $\eta(0)=0$. If ${\bf \Sigma}'$ is trivial, then $Pf=f(0).$ If we are in the second case, then Theorem \ref{1.2} implies that $\eta$ is a finite Blaschke product and $\|\eta'\|_\infty =1 $. It follows from Proposition \ref{3.5} that $\eta=cz$ for some unimodular constant $c$, and hence $P$ is the identity.
	\end{proof} 

	It seems that the assumption $Pz^n \in H^\infty(\T)$ in Theorem \ref{3.8} arises from technical reasons. We conjecture that this assumption can be dropped. Specifically, we conjecture that if $P$ is a contractive projection on $H^p(\T)$ with $0<p<1$ satisfying $P\one=\one$, then it automatically maps bounded functions to bounded functions, and therefore $Pf=f(0)$ or $P$ is the identity.

\section{contractive projection sets for $H^p(\mathbb{T}^d),~0<p<1$.}

In this section, we apply the results in previous sections to study contractive idempotent coefficient multipliers on $H^p(\T^d)$ for $0<p<1$ and $1\leq d\leq \infty$.
Suppose that $\Gamma $ is a nonempty subset of $\mathbb N_0^{(d)}$ and $P_\Gamma$ is the associated idempotent coefficient multiplier defined by \eqref{eq1.1}.
As we mentioned in the introduction, Brevig, Ortega-Cerd\`{a} and Seip \cite{BOS} completely determined when $P_{\Gamma}$ extends to a contraction on $H^p(\T^d)$ for $1\leq p\leq \infty$. In their work, such $\Gamma$ is termed a contractive projection set for $H^p(\mathbb{T}^d)$. Our main aim in this section is to characterize contractive projection sets for $H^p(\mathbb{T}^d)$ for the remaining case $0<p<1$.

Observe that both $\{0\}$ (corresponding to $Pf=f(0)$) and $\mathbb{N}_0$ (corresponding to $P$ being the identity) are contractive projection sets for $H^p(\mathbb{T})$. In fact, there are no others when $0<p<1$.

\begin{theorem}\label{4.1}
	Suppose that $0 < p < 1$ and $\Gamma $ is a contractive projection set for $H^p(\T)$. Then $\Gamma=\{0\}$ or $\Gamma=\mathbb{N}_0$.
\end{theorem}

\begin{proof}
Suppose $k$ is the smallest integer in $\Gamma$. We claim that $k=0$. Once this claim is established, it follows that $P_\Gamma \one = \one$, and the conclusion is an immediate consequence of Theorem \ref{3.8}. Indeed, now assume $k\neq 0$ by contradiction. Define
	$$
	Qf = \bar{z}^k P_{\Gamma}\big(z^{k} f(z)\big), \quad f\in X,
	$$
	where $$X:= \{f\in L^p(\T): z^k f \in H^p(\T)\}.$$
	Clearly, $X$ is a closed subspace of $L^p(\T)$ and $Q$ extends to a contractive projection on $X$ satisfying $Q\one =\one$. According to Theorem \ref{1.1}, there exists a sub $\sigma$-algebra ${\bf \Sigma}'$ such that 
	\begin{align}\label{eq4.1}
		Q(z^{n-k})=\E(z^{n-k}|{\bf \Sigma}'), \quad n\geq 0.
	\end{align}
	Observe that an non-negative integer $n$ we have $$n\in \Gamma \quad \Leftrightarrow \quad P_\Gamma z^n=z^n \quad  \Leftrightarrow\quad  Q(z^{n-k})=z^{n-k}.$$ Combining this with \eqref{eq4.1} we deduce that 
	\begin{align}\label{eq4.2}
		\Gamma = \{n\geq 0: z^{n-k} \,\, {\rm is}\,\, {\bf \Sigma}'-{\rm measurable}\}.
	\end{align}
	Note that $Q=P_{\Gamma-k}$ on $H^p(\T)$, so by Theorem \ref{3.8}, $\Gamma-k=\{0\}$ or $\Gamma-k=\mathbb N_0$. Equivalently, $\Gamma=\{k\}$ or $\Gamma=\{k,k+1,\ldots\}$. 
	
	If $\Gamma = \{k, k+1, \ldots\}$, then \eqref{eq4.2} implies that the coordinate function $z$ is ${\bf \Sigma}'$-measurable. Hence $\Gamma = \mathbb{N}_0$, which contradicts the assumption that $k \neq 0$.

	It remains to exclude the case $\Gamma = \{k\}$. To this end, it is sufficient to show that $c(k, p) > 1$, where $c(k, p)$ denotes the norm of the $k$-th Taylor coefficient functional on $H^p(\T)$. For $f \in H^p(\mathbb{T})$, define $f_k(z) = f(z^k)$. Then $\|f\|_p = \|f_k\|_p$ and $f'(0) = f_k^{(k)}(0)/k !$. Thus $c(k, p) \geq c(1, p)$. Now the proof is completed as the explicit value
	$$
	c(1, p) = \sqrt{\frac{2}{p}} \left(1 - \frac{p}{2}\right)^{\frac{1}{p} - \frac{1}{2}}>1
	$$
	was established by Brevig and Saksman \cite{BS}.
\end{proof}

For $1\leq d \leq \infty$ and a subset $J\subseteq \{1, 2, \cdots, d\}$ (if $d=\infty$, $J\subseteq \mathbb{N}$),  we define $\mathbb{N}_J=\{\alpha\in \mathbb N_0^d: \alpha_j =0\,\, {\rm for\,\, all }\,\, j\in J\}$. For an analytic polynomial $f$ and $p>0$, the function $|f|^p$ is plurisubharmonic. Using the sub-mean value principle, it is easy to verify that $\mathbb N_J$ is a contractive projection set for $H^p(\T^d)$. Actually, when $0< p<1$,  the collection $\{\mathbb N_J: J \,\, {\rm runs\,\, over\,\, all\,\, subsets } \} $ exhausts all contractive projection sets for $H^p(\T^d)$.

\begin{theorem}\label{4.2}
	Suppose that $0<p<1$ and $1\leq d\leq \infty$, and $\Gamma$ is a contractive projection set for $H^p(\mathbb T^d)$. Then $\Gamma=\mathbb{N}_J$ for some subset $J$.
\end{theorem}
\begin{proof}
Let $J=\{j: \alpha_j=0,~\forall \alpha\in\Gamma\}$. We will show that $\Gamma=\mathbb{N}_{J}$. Clearly $\Gamma\subseteq\mathbb{N}_{J}$. Next, we prove the converse inclusion. Let $e^{(i)}$ be the canonical  basis of $\mathbb N_0^d$ determined by $e^{(i)}_i =1$ and $e^{(i)}_j =0$ for $j\neq i$. For a $\beta\in \mathbb N_0^d$, we define $\beta^{[i]}$
by $\beta^{[i]}_i=0$ and $\beta^{[i]}_j= \beta_{j}$ for $j\ne i$.	In what follows,  we may assume $\Gamma\neq\{0\}.$

We first claim that if $\beta \in \Gamma$ with $\beta_i\neq 0$, then $\beta^{[i]}+ke^{(i)}\in\Gamma$ for all $k\in \mathbb N_0$. Indeed, for any analytic polynomial $q$ in one single variable, one can verify that the function $$Pq(z):= \bar{z}^{\beta^{[i]}} P_\Gamma \big(z^{\beta^{[i]}} q(z_i)\big)$$ only depends on the variable $z_i$. Therefore, $P$ yields a contractive projection on $H^p(\T)$ with the contractive projection set $\{k \in \mathbb N_0: \beta^{[i]}+k e^{(i)} \in \Gamma\}$. Now our claim follows from Theorem \ref{4.1} because $\beta_i\ne 0$ and $\beta_i$ belongs to this set.

Take a $\beta=(\beta_j)_{j\geq 1} \in \Gamma$ with $ \beta\ne 0$. Since there are only  finitely many $j$ such that $\beta_j\ne 0$,  iteratively applying the previous claim shows $\sum_{\beta_j\ne 0} k_j e^{(j)} \in \Gamma$ for all $k_j\in \mathbb N_0.$  In particular, $0\in \Gamma$. Together this with Theorem \ref{1.1},  $\Gamma$ is a sub-semigroup of $\mathbb N_0^d$.  For  $\alpha=(\alpha_j)_{j\geq 1}\in \mathbb{N}_{J}$, if  $\alpha_l\ne 0$,  then by definition  there exists a $\beta\in \Gamma$ such that $\beta_l\ne 0$, and hence $\alpha_l e^{(l)}\in  \Gamma$. Since $\Gamma$ is a sub-semigroup  of $\mathbb N_0^d$,  we see that    $\alpha = \sum_{l}\alpha_l e^{(l)}\in \Gamma$, completing the proof.
	\end{proof}

We end this paper by reformulating Theorem \ref{4.2} for the Hardy spaces of Dirichlet series $\mathscr{H}^p$ via the Bohr transform, which corresponds to the case $d=\infty$ in Theorem \ref{4.2}. The Hardy space of Dirichlet series $\mathscr{H}^p$ is defined as the closure of all Dirichlet polynomials $f(s) = \sum_{n=1}^N a_n n^{-s}$ under the norm (or quasi-norm when $0 < p < 1$),
\[
\|f\|_{\mathscr{H}^p}^p = \lim_{T \to \infty} \frac{1}{2T} \int_{-T}^T |f(it)|^p \, dt.
\]
It is known that the elements of $\mathscr{H}^p$ are indeed Dirichlet series absolutely converging on the right half plane $ {\rm Re}\, s> 1/2$. Let $p=(p_1,p_2, \ldots )$ be the ordered prime numbers. According to the fundamental theorem of arithmetic, each integer $n\in \mathbb N$ can be uniquely decomposed as $$n=p^{\kappa(n)}:=\prod_{j=1}^\infty p_j^{\kappa_j}, \quad \kappa(n)\in \mathbb N_0^d.$$ By the Birkhoff-Oxtoby ergodic theorem, the Bohr transform
\[
\mathscr{B}: \mathscr{H}^p \rightarrow H^p(\mathbb{T}^\infty), \quad f(s) = \sum_{n=1}^{\infty} a_n n^{-s} \mapsto \mathscr{B}f(z) = \sum_{n=1}^{\infty} a_n z^{\kappa(n)},
\]
establishes an isometric isomorphism between $\mathscr{H}^p$ and $H^p(\mathbb{T}^\infty)$.

\vspace{1mm}
A nonempty subset $\Gamma \subseteq \mathbb{N}$ is termed a contractive projection set for $\mathscr{H}^p$ if the operator $P_\Gamma: \sum_{n=1}^\infty a_n n^{-s} \mapsto \sum_{n \in \Lambda} a_n n^{-s},$ defined on the space of Dirichlet polynomials,  extends to a contraction on $\mathscr{H}^p$.
\begin{theorem}\label{4.3}
	Suppose that $0<p<1$. Then $\Gamma$ is a contractive projection set for $\mathscr{H}^p$ if and only if there is a subset $J$ of all prime numbers such that
	$\Gamma=\{n \in\mathbb N : {\rm all \,\, prime\,\, divisors  \,\,of \,\,} n {\rm\,\, are \, \,in\,\,} J\}$.
\end{theorem}

\section*{Acknowledgements}
This work was supported by the National Key R\&D Program of China (2024YF A1013400) and the National Natural Science Foundation of China (Grant No. 12231005). Dilong Li was supported by the  National Natural Science Foundation of China (Grant No. 124B2005).


\end{document}